\documentclass[a4paper,12pt]{article}

\usepackage{amsmath}
\usepackage{amssymb}
\usepackage{amsthm}

\newcommand{\NN}{{\mathbb N}}
\newcommand{\ZZ}{{\mathbb Z}}
\newcommand{\QQ}{{\mathbb Q}}
\newcommand{\RR}{{\mathbb R}}
\newcommand{\PP}{{\mathbb P}}

\theoremstyle{definition}
\newtheorem{thm}{Theorem}
\newtheorem{lem}[thm]{Lemma}
\newtheorem{cor}[thm]{Corollary}

\newtheorem*{rem}{Remark}

\begin{document}

\title{Existence of primes in the interval \( [15x,16x] \) \\
-- An entirely elementary proof --}
\author{%
Hiroki Aoki\thanks{aoki\_hiroki\_math@nifty.com, Faculty of Science and Technology, Tokyo University of Science.}
\and
Riku Higa\thanks{6123702@ed.tus.ac.jp, Faculty of Science and Technology, Tokyo University of Science.}
\and
Ryosei Sugawara\thanks{Graduated from Faculty of Science and Technology, Tokyo University of Science.}
}

\maketitle

\begin{abstract}
In this paper,
we give a short and entirely elementary proof of
the proposition \lq\lq For any positive integer \( N \),
there exists a real number \( L \) such that
for any real number \( x \geqq L \),
there are at least \( N \) primes in the interval \( [kx, (k+1)x] \)''
for \( k \leqq 15 \).
Our proof is based on
the idea of the proof by Erd\"{o}s~\cite{Erd} for \( k=1 \)
and its improvement
by Hitotsumatsu~\cite{Hit} and by Sainose~\cite{Sai} for \( k=2 \).
In the case of \( k=3 \) and \( k=4 \), the method is very
similar to the case of \( k=2 \), however,
in the case of \( k \geqq 5 \), we need new idea to complete the proof.
\end{abstract}

\noindent
{\bf{Keywords:}} Prime numbers. \\
{\bf{Mathematics Subject Classification:}} 11A41, 11N05.

\section{Introduction}

Let \( k \) be a positive integer.
In this paper, we treat the proposition \( P(k) \):
\lq\lq For any positive integer \( N \),
there exists a positive integer \( L_{N,k} \) such that
for any real number \( x \geqq L_{N,k} \),
there are at least \( N \) primes in the interval \( [kx, (k+1)x] \).''
In the case of \( k=1 \), this is very near to an old question
\lq\lq Is there a prime number in the interval \([n, 2n] \) for
any \( n \in \NN \)?'', which was conjectured yes
by Joseph Bertrand in 1845 and proved by Pafnuty Chebyshev in 1852.
Hence this old question is called the {\bf{Bertrands postulate}}
or the {\bf{Bertrand-Chebyshev theorem}}.
Later, Srinivasa Ramanujan~\cite{Ram} gave a short proof
of Bertrand-Chebyshev theorem in 1919.
His proof includes that \( P(1) \) is true.
After then,
Jitsuro Nagura~\cite{Nag} proved \( P(k) \) for \( k \leqq 5 \) in 1952,
by improving Ramanujan's technique.
Today, since the Prime Number Theorem (PNT) has already been proved,
we know that \( P(k) \) is true for any \( k \in \NN \) (c.f.~\cite{HW}).
However, PNT is one of the deep theorem in number theory
and its proof is not short.
Thus, even recently, the work
to find more simple proof of \( P(k) \) is still in progress.
A remarkable proof of Bertrand-Chebyshev theorem
was given by Paul Erd\"{o}s~\cite{Erd} in 1932.
His proof is not only short but also very elementary enough that
usual high school students can understand.
Improving his technique,
Mohamed El Bachraoui~\cite{Bac} gave an elementary proof
of \( P(2) \) in 2006.
His improvement is quite useful for small \( k \).
Using his idea,
Andy Loo~\cite{Loo} proved \( P(3) \) in 2011 by very similar way
(but not entirely elementary), and then,
Kyle D. Balliet~\cite{Bal} proved \( P(k) \) for \( k \leqq 8 \) in 2015.
However, Balliet~\cite{Bal} remarked that his proof
is not entirely elementary for \( k \geqq 5 \).
In his proof for \( k \geqq 5 \), he used shape estimations of
some functions, which seems not to be elementary.

It does not seem very well known but some Japanese high school teachers
gave another amazing results on this proposition.
In 2011, Japanese mathematician Shin Hitotsumatsu~\cite{Hit} introduced
a very elementary and simplified proof of Bertrand-Chebyshev theorem
according to the idea of Erd\"{o}s on \lq\lq S\={u}ken Tsushin'',
which is a journal of educational practice of mathematics
written in Japanese and mainly for high school teachers.
After his work, in 2013,
Shigenori Tochiori~\cite{Toc} improved his idea
and Ichiro Sainose~\cite{Sai} applied their method to \( P(2) \).
In 2018, Naoya Suzuki~\cite{Suz} improved Sainose's proof of \( P(2) \).
Their proofs are entirely elementary.
Recently in 2023,
the third author showed \( P(5) \) in his master's thesis
according to this way with some new idea.
In this paper, we extend his result
to \( P(k) \) for \( k \leqq 15 \) by refining his idea.
Our proof is also entirely elementary.

\section{Preliminaries}

Throughout this paper, \(
\RR , \; \QQ , \; \ZZ , \; \NN \) and \( \PP \) denote
the sets of all real numbers, rational numbers, integers,
positive integers and primes, respectively.
The symbols \( x,y \) always should be real numbers.
The symbols \( m,n,s,t,u \) always should be a natural numbers.
The symbol \( p \) always should be a prime.
We denote by \( \lfloor x \rfloor \) and by \( \lceil x \rceil \)
the greatest integer less than or equal to \( x \) and
the least integer greater than or equal to \( x \), respectively.
In this section, we show some lemmas for later use.

\subsection{Our direction}

Let \( k \in \NN \).
The proposition we treat in this paper is:
\[
 P(k): \;
 \forall N \in \NN , \;
 \exists L_{k,N} \in \RR ; \; \forall x \geqq L_{k,N} ,
 \; \# ( \PP \cap [kx, (k+1)x] ) \geqq N.
\]
Here, let \( a \in \NN \) and
we consider one more proposition:
\[
 P^* (k,a): \;
 \forall N \in \NN , \;
 \exists L_{k,a,N}^* \in \NN ; \; \forall n \geqq L_{k,a,N}^* ,
 \; \# ( \PP \cap (kan, (k+1)an] ) \geqq N.
\]
\begin{lem}
Two propositions \( P^* (k,a) \) and \( P(k) \) are equivalent.
\end{lem}
\begin{proof}
\( P(k) \Rightarrow P^* (k,a) \) : It is easy to see
that we can take \( L_{k,a,N}^* = \lceil \frac{L_{k,N+1}}{a} \rceil \).
\\
\( P^* (k,a) \Rightarrow P(k) \) : Since \( kx
-ka \lfloor \frac{x}{a} \rfloor  < ka \),
the number of primes in the interval \(
\left( ka \lfloor \frac{x}{a} \rfloor , kx \right) \) is
equal or less than \( ka \).
Hence we can take \( L_{k,N} =a L_{k,a,N+ka}^* \).
\end{proof}
In this paper, we will show \( P^* (k,a) \) instead of \( P(k) \) for
some pairs of \( k \) and \( a \).

\subsection{Stirling's formula}

We need to use the Stirling's formula, an approximation for factorials.
However, some well-known versions of Stirling's formula
seem not to be proven by elementary way,
although they are strong and useful.
Here, we show the following elementary lemma,
which is a very weak version
of the Stirling's formula.

\begin{lem}[Stirling's formula, cf. Sainose~\cite{Sai}]
\label{lemm:str}
For any \( n \in \NN \), the following inequality holds:
\[
 n^n e^{-n+1} \leqq n! \leqq n^{n+1} e^{-n+1}
 .
\]
It becomes an equality if and only if \( n=1 \).
\end{lem}
\begin{proof}
We may assume \( n \geqq 2 \).
Since \( \log (n!) = \sum_{x=2}^n \log x \), we have
\[
 \int_{1}^{n} \log x \, dx
 <
 \log (n!)
 < \log n + \int_{1}^{n} \log x \, dx
 .
\]
As \( \int_{1}^{n} \log x \, dx = n \log n -n+1 \), we have
\[
 n \log n -n+1 < \log (n!) < \log n + n \log n -n+1 
 .
\]
\end{proof}

\begin{rem}
The following inequality is a strong version of the Stirling's formula
given by Robbins~\cite{Rob}:
\[
 \sqrt{2 \pi } \, n^{n+ \frac{1}{2}} e^{-n+\frac{1}{12n+1}}
 < n! <
 \sqrt{2 \pi } \, n^{n+ \frac{1}{2}} e^{-n+\frac{1}{12n}}
 .
\]
Generally, his proof is not regarded as elementary.
Balliet~\cite{Bal} proved \( P(k) \) for \( 5 \le k \le 8 \) by
using this inequality.
\end{rem}

\subsection{Legendre's formula}

Another elementary tool we need in this paper is
the Legendre's formula, which gives
the prime factorization of factorials.
Let
\[
 n! = \prod_{p \in \PP} p^{\nu_p (n!)}
\]
be a prime factorization of \( n! \).
Namely, \( \nu_p (n!) \) denotes the exponent
of the largest power of \( p \) that divides \( n! \).

\begin{lem}[Legendre's formula]
\label{lemm:leg}
For any \( p \in \PP \) and \( n \in \NN \), we have
\[
 \nu_p (n!)=
 \sum_{j=1}^{\infty} \left\lfloor \frac{n}{p^j} \right\rfloor
 .
\]
\end{lem}
\begin{proof}
\[
 \nu_p (n!)
 =
 \sum_{j=1}^{\infty}
 \# \{ \; m \in \NN \; | \; m \leqq n , \; p^j | m \; \}
 =
 \sum_{j=1}^{\infty} \left\lfloor \frac{n}{p^j} \right\rfloor
\]
\end{proof}

\subsection{Prime counting function and primorial}

For any \( x \in \RR \),
we denote by \( \pi (x) \) and \( \Theta (x) \) the
prime counting function and the primorial of \( x \) by \( \Theta (x) \),
respectively.
Namely, \( \pi (x) \) is the number of primes
less than or equal to \( x \) and \( \Theta (x) \) is
the product
extending over all prime numbers \( p \) that are
less than or equal to \( x \):
\[
 \pi (x) := \# \{ \; p \in \PP \; | \; p \leqq x \; \} , \qquad
 \Theta (x) := \prod_{p \leqq x} p
 .
\]
For simplicity, we define \( \Theta (x)=1 \) for \( x<2 \).

\begin{lem}[cf. Hitotsumatsu~\cite{Hit}]
\label{lemm:pie}
For any \( x \geqq 36 \), the following inequality holds:
\[
 \pi (x) \leqq \frac{x}{3}
 .
\]
\end{lem}
\begin{proof}
Because every odd prime number
other than \( 3 \) are \( 1 \) or \( 5 \) modulo \( 6 \),
we have \( \pi (3(n+1))- \pi (3n) \leqq 1 \) for \( n \geqq 2 \).
Since \( \pi (36)= 11 = \frac{36}{3} -1 \), this lemma holds.
\end{proof}

\begin{lem}[cf. Erd\"{o}s~\cite{Erd}, Tochiori~\cite{Toc}]
\label{lemm:the}
For any \( x \geqq 2 \), the following inequality holds:
\[
 \Theta (x) \leqq \frac{1}{8} \cdot 4^x
 .
\]
\end{lem}
\begin{proof}
It is enough to show the case \( x=n \in \NN \).
We can easily see that the case \( n=2 \) is true.
Because \( \frac{\Theta (2n-1)}{\Theta (n)} \leqq \binom{2n-1}{n-1} \leqq
\frac{1}{2} (1+1)^{2n-1} = 4^{n-1} \),
we have \( \Theta (2n)= \Theta (2n-1) \leqq 4^{n-1} \Theta (n) \).
Hence, by induction, it follows that this lemma holds.
\end{proof}

\begin{rem}
This lemma is a very rough estimation of the primorial.
For example, if we start from \( \frac{(6n)!}{(3n)!(2n)!n!} \),
we have \( \Theta (x)< 3.4^x \) by similar way.
Hanson~\cite{Han} showed \( \Theta (x)< 3^x \) by improving this idea.
But Balliet~\cite{Bal} proved \( P(k) \) for \( 5 \le k \le 8 \) by
using much sharper estimation, which was proved by modern deep mathematics.
Actually, we know
that \( \lim_{x \to \infty} (\Theta (x))^{\frac{1}{x}} =e \).
As for the discussion of this paper,
the sharpness of the estimation of primorial
is closely related to the difficulty of showing \( P(k) \).
Basically, we use {\bf{Lemma \ref{lemm:the}}} as the estimation
of the primorial in this paper.
\end{rem}

\section{Main Theory}

\subsection{General consideration}

\label{sec:genc}

In this subsection we consider generalization of previous researches
(c.f.~\cite{Bac,Bal,Erd,Hit,Sai,Toc,Suz}).
To prove \( P^* (k,a) \),
let \( a_1 := (k+1)a , \; b_1 := ka \) and
take \( a_2 , \dots , a_s \in \NN \) and \(
b_2 , \dots , d_t \in \NN \) such that
\[
 b_1 \geqq a_2 \geqq \cdots \geqq a_s
 , \quad
 b_1 \geqq b_2 \geqq \cdots \geqq b_t
\]
and
\begin{equation}
\label{eq:sscd}
 a_1 + a_2 + \cdots + a_s
 =
 b_1 + b_2 + \cdots + b_t
\end{equation}
hold.
We put
\[
 F_k (n):= \frac{
 ( a_1 n)!( a_2 n)! \cdots ( a_s n)!
 }{
 ( b_1 n)!( b_2 n)! \cdots ( b_t n)!
 }
 \in \QQ
 .
\]
Let
\[
 F_k (n) = \prod_{p \in \PP} p^{\nu_{k,p} (n)}
\]
be a prime factorization of \( F_k (n) \).
By {\bf{Lemma \ref{lemm:leg}}}, we have
\[
 \nu_{k,p} (n) = \sum_{j=1}^{\infty} G_k \left( \frac{n}{p^j} \right)
 \in \ZZ
 ,
\]
where
\[
 G_k \left( x \right) :=
 \lfloor a_1 x \rfloor
 + \lfloor a_2 x \rfloor
 + \cdots
 + \lfloor a_s x \rfloor
 - \lfloor b_1 x \rfloor
 - \lfloor b_2 x \rfloor
 - \cdots
 - \lfloor b_t x \rfloor
 .
\]
By the condition (\ref{eq:sscd}), \( G_k (x) \) depends
only on the fractional part of \( x \).
Let
\[
 L:= {\mathrm{LCM}} ( a_1 , a_2 , \dots , a_s , b_1 , b_2 , \dots , b_t )
 .
\]
Then the value of \( G_k (x) \) is determined
from the integral part of \( Lx \).
Therefore we can define the maximum of \( G_k (x) \), namely
\[
 M:= \max \left\{ \; G_k \left( \frac{j}{L} \right) \; \middle| \;
 j \in \{ 0,1, \dots ,L-1 \} \; \right\}
 .
\]
We can show the following properties easily.
\begin{itemize}
\item
\( G_k (x)=0 \) if \( 0 \leqq x< \frac{1}{a_1} \).
\item
\( G_k (x)=1 \) if \( \frac{1}{a_1} \leqq x< \frac{1}{b_1} \).
\item
\( G_k (x) \leqq 0 \) if \( \frac{1}{b_1} \leqq x< \min \left\{
\frac{2}{a_1} , \frac{1}{a_2} \right\} \).
\end{itemize}
Let
\[
 c_i := \min \left\{ x \geqq \frac{1}{b_1} \; \middle| \;
 G_k (x) \geqq i \right\}
 \quad
 (i \in \NN , \; i \leqq M)
 .
\]
This means that \( G_k (x)<i \) for any \( \frac{1}{b_1} \leqq x< c_i \).
Hence \( G_k \left( \frac{n}{p} \right) <i \) when \(
\frac{n}{ c_i } <p \leqq b_1 n \).

\bigskip

Now we assume that \( n \) is sufficiently large and
decompose \( F_k (n) \) as follows:
\[
 F_k (n)
 =
 \underbrace{ \left(
 \prod_{b_1 n<p \leqq a_1 n} p
 \right) }_{F_{k,{\textrm{A}}} (n)}
 \underbrace{ \left(
 \prod_{ \sqrt{a_1 n} <p \leqq b_1 n } p^{G_k \left( \frac{n}{p} \right)}
 \right) }_{F_{k,{\textrm{B}}} (n)}
 \underbrace{ \left(
 \prod_{ p< \sqrt{a_1 n} } p^{\nu_{k,p} (n)}
 \right) }_{F_{k,{\textrm{C}}} (n)}
 .
\]

\medskip

\noindent
{\bf{(A)}}
As for \( F_{k,{\textrm{A}}} (n) \), we have
\[
 F_{k,{\textrm{A}}} (n) \leqq ( a_1 n)^{N_k (n)}
 ,
\]
where
\[
 N_k (n) := \pi ( a_1 n) - \pi ( b_1 n)
 ,
\]
which is the number of primes
in the interval \( (b_1 n, a_1 n] \).

\medskip

\noindent
{\bf{(B)}}
As for \( F_{k,{\textrm{B}}} (n) \), we have
\[
 F_{k,{\textrm{B}}} (n)
 \leqq
 \prod_{p \leqq b_1 n } p^{G_k \left( \frac{n}{p} \right)}
 \leqq
 \prod_{i=1}^M \left( \prod_{p \leqq \frac{n}{c_i}} p \right)
 =
 \prod_{i=1}^M \Theta \left( \frac{n}{c_i} \right)
 .
\]
Hence, by {\bf{Lemma \ref{lemm:the}}}, we have
\[
 F_{k,{\textrm{B}}} (n)
 \leqq
 \prod_{i=1}^M \Theta \left( \frac{n}{c_i} \right)
 \leqq
 \frac{4^{cn}}{8^M} \qquad
 \left( \; c:= \sum_{i=1}^{M} \frac{1}{c_i} \; \right)
 .
\]

\medskip

\noindent
{\bf{(C)}}
As for \( F_{k,{\textrm{C}}} (n) \),
by {\bf{Lemma \ref{lemm:pie}}}, we have
\[
 F_{k,{\textrm{C}}} (n) \leqq
 \prod_{ p< \sqrt{a_1 n} } p^{M \log_p ( a_1 n)}
 =( a_1 n)^{M \pi ( \sqrt{a_1 n} )}
 \leqq ( a_1 n)^{\frac{M \sqrt{ a_1 n}}{3}}
 .
\]

\medskip

Summarizing the above,
we have an upper bound of \( F_k (n) \):
\[
 F_k (n) \leqq
 ( a_1 n)^{N_k (n)}
 \left( \frac{4^{cn}}{8^M} \right)
 ( a_1 n)^{\frac{M \sqrt{ a_1 n}}{3}}
 .
\]
On the other hand,
by {\bf{Lemma \ref{lemm:str}}},
we have a lower bound of \( F_k (n) \):
\[
 F_k (n)
 \geqq
 \frac{\bigl( a_1^{a_1} a_2^{a_2} \cdots a_s^{a_s} \bigr)^n
 e^{s-t}}{
 \bigl( b_1^{b_1} b_2^{b_2} \cdots b_t^{b_t} \bigr)^n
 \left( b_1 b_2 \cdots b_t \right) n^t }
 .
\]
Hence
\[
 \frac{\bigl( a_1^{a_1} a_2^{a_2} \cdots a_s^{a_s} \bigr)^n
 e^{s-t}}{
 \bigl( b_1^{b_1} b_2^{b_2} \cdots b_t^{b_t} \bigr)^n
 \left( b_1 b_2 \cdots b_t \right) n^t }
 \leqq
 ( a_1 n)^{N_k (n)}
 \left( \frac{4^{cn}}{8^M} \right)
 ( a_1 n)^{\frac{M \sqrt{ a_1 n}}{3}}
\]
holds and therefore we have
\[
 N_k (n)
 \geqq
 \frac{
 n \left( \log \left( \frac{ a_1^{a_1} a_2^{a_2} \cdots a_s^{a_s}}{
 b_1^{b_1} b_2^{b_2} \cdots b_t^{b_t}} \right)
 - c \log 4 \right)
 + o(n)
 }{
 \log a_1 + \log n
 }
 .
\]
Consequently, we have the following theorem.

\begin{thm}
\label{thm:main}
If the condition
\begin{equation}
\label{eq:main}
 \left( \sum_{i=1}^s a_i \log a_i \right) -
 \left( \sum_{i=1}^t b_i \log b_i \right) >
 c \log 4
 ,
\end{equation}
holds, then
we have \( \displaystyle \lim_{n \to \infty} N_k (n) =+ \infty \) and
therefore \( P^* (k,a) \) is true.
\end{thm}

\subsection{Previous researches}

In this subsection we summarize previous researches
from the view point of the previous subsection.

\begin{itemize}
\item
{\bf{In the case of \( k=1 \):}}\\
Essentially, Erd\"{o}s~\cite{Erd} proved \( P^* (1,1) \) by using
\[
 F_1^* (n) := \frac{(2n)!}{n!n!}
 .
\]
This is the case \( s=1, \; a_1 =2, \; t=2, \; b_1 = b_2 =1 \).
By easy calculation,
we have \( M=1, \; c_1 = \frac{3}{2} , \; c= \frac{2}{3} \) and
therefore the condition (\ref{eq:main}) holds.
Also,
Hitotsumatsu~\cite{Hit} and Tochiori~\cite{Toc} used the same function.
\item
{\bf{In the case of \( k=2 \) (a):}}\\
Essentially, Bachraoui~\cite{Bac} proved \( P^* (2,2) \) by using
\[
 F_2^* (n) := \frac{(6n)!n!}{(4n)!(3n)!}
 .
\]
In this case we have \( M=1, \; c_1 = \frac{7}{6} , \; c= \frac{6}{7} \) and
therefore the condition (\ref{eq:main}) holds.
Also, Suzuki~\cite{Suz} used the same function.
\item
{\bf{In the case of \( k=2 \) (b):}}\\
Sainose~\cite{Sai} proved \( P^* (2,2) \) by using
\[
 F_2 (n) := \frac{(6n)!(2n)!}{(4n)!(3n)!n!}
 .
\]
In this case we have \( M=1, \; c_1 = \frac{1}{2} , \; c=2 \) and
therefore the condition (\ref{eq:main}) holds.
\item
{\bf{In the case of \( k=4 \):}}\\
Essentially, Balliet~\cite{Bal} proved \( P^* (4,6) \) by using
\[
 F_4 (n) :=
 \frac{(30n)!(12n)!(8n)!(3n)!(2n)!}{(24n)!(15n)!(10n)!(6n)!}
 .
\]
In this case
we have \( M=1, \; c_1 = \frac{23}{30} , \; c= \frac{30}{23} \) and
therefore the condition (\ref{eq:main}) holds.
\item
{\bf{In the case of \( k \geqq 5 \):}}\\
In the case of \( k \geqq 5 \),
Balliet~\cite[p.17]{Bal} insisted that his method is inconclusive.
Since
\[
  F_1^* (n) = \binom{2n}{n}
  , \quad
  F_2^* (n/2) = \frac{\binom{3n}{2n}}{\binom{3n/2}{2n/2}}
  \quad \text{and} \quad
  F_4 (n/6) = \frac{\binom{5n}{4n}}{\binom{5n/2}{4n/2}\binom{5n/3}{4n/3}}
 ,
\]
he thought the next one will be
\[
 F_5 (n/30) = \frac{\binom{6n}{5n}}{
 \binom{6n/2}{5n/2}\binom{6n/3}{5n/3}\binom{6n/5}{5n/5}}
 .
\]
However, \( F_5 (n/30)<1 \) for large \( n \).
This means that the condition (\ref{eq:main}) does not hold
without any calculation of \( c \).
\end{itemize}

Here we remark that all the above success cases are for \( M=1 \).
Since {\bf{Theorem \ref{thm:main}}} can be
applied for any \( M \), our argument is already new.
Intuitively, because
the right hand side \( (\text{RHS}) \) of
the condition (\ref{eq:main}) may become larger when \( M>1 \),
there is considerable resistance to setting \( M>1 \).
However, in practice, as we will see below,
this theorem is quite useful even when \( M \) is large.

\subsection{Our new idea}

Let \( P \subset \PP \) be a finite set and \( Q \) be
the product of all primes in \( P \).
Namely,
\[
 Q:= \prod_{p \in P} p
 .
\]
In this subsection we put
\[
 F_k (n) =
 F_k^* (P;n) :=
 \prod_{m|Q}
 \left( \frac{
 \left( \frac{(k+1)nQ}{m} \right) !
 }{
 \left( \frac{knQ}{m} \right) ! \left( \frac{nQ}{m} \right) !
 } \right)^{\mu (m)}
 =
 \prod_{p \in \PP} p^{\nu_{k,p} (P;n)}
 ,
\]
where \( \mu \) is the M{\"{o}}bius function,
and we apply the argument in {\bf{\S\ref{sec:genc}}}.
Then we have
\[
 \nu_{k,p} (P;n) = \sum_{j=1}^{\infty} G_k \left( P; \frac{n}{p^j} \right)
 ,
\]
where
\[
 G_k \left( P;x \right) :=
 \sum_{m|Q} \mu (m) \left(
 \left\lfloor \frac{(k+1)Q}{m}x \right\rfloor
 -
 \left\lfloor \frac{kQ}{m}x \right\rfloor
 -
 \left\lfloor \frac{Q}{m}x \right\rfloor
 \right)
 .
\]

\begin{lem}
As for the above \( F_k^* (P;n) \),
the left hand side \( (\text{LHS}) \) of
the condition (\ref{eq:main}) is
\[
 \varphi (Q) \log \left( \frac{(k+1)^{k+1}}{k^k} \right)
 ,
\]
where \( \varphi \) is the Euler's totient function.
\end{lem}
\begin{proof}
First we put
\[
 Z(x) := \sum_{m|Q} \mu (m) \frac{x}{m} \log \frac{x}{m}
 .
\]
Then we have
\begin{align*}
 Z(x)
 & =
\sum_{m|Q} \mu (m) \frac{x}{m} ( \log x - \log m)
 \\
 & =
 \sum_{m|Q} \mu (m) \frac{x}{m} \log x
 - \sum_{m|Q} \mu (m) \frac{x}{m} \log m
 \\
 & =
 \left( \sum_{m|Q} \frac{\mu (m)}{m} \right) x \log x
 - \left( \sum_{m|Q} \frac{\mu (m) \log m}{m} \right) x
 \\
 & =
 \frac{\varphi(Q)}{Q} x \log x
 - \left( \sum_{m|Q} \frac{\mu (m) \log m}{m} \right) x
 .
\end{align*}
Therefore, we have
\begin{align*}
 (\text{LHS})
 & =
 \sum_{m|Q} \mu (m) \left(
 \frac{(k+1)Q}{m} \log \frac{(k+1)Q}{m}
 -
 \frac{kQ}{m} \log \frac{kQ}{m}
 -
 \frac{Q}{m} \log \frac{Q}{m}
 \right)
 \\
 & =
 \sum_{m|Q} \left( Z((k+1)Q)-Z(kQ)-Z(Q) \right)
 \\
 & =
 \frac{\varphi(Q)}{Q}
 \left( (k+1)Q \log ((k+1)Q) -kQ \log (kQ) - Q \log Q \right)
 \\
 & =
 \varphi (Q) \log \left( \frac{(k+1)^{k+1}}{k^k} \right)
 .
\end{align*}
\end{proof}

Since \( \varphi (Q) = \prod_{p \in P} (p-1) \),
we can compute the value of (LHS) immediately.
To compute the value of (RHS), the following lemma and corollary are
useful to find \( c_i \)'s.

\begin{lem}
\label{lemm:fdci}
We have
\[
 G_k \left( P;x \right) = \sum_{(u,Q)=1}
 \left( \chi_{u} ((k+1)Qx) - \chi_{u} (kQx) - \chi_{u} (Qx) \right)
 ,
\]
where we put
\[
 \chi_y (x) := \left\{
 \begin{array}{ll}
 1 & (x \geqq y)\\
 0 & (x<y)\\
 \end{array}
 \right.
 .
\]
\end{lem}
\begin{proof}
Since
\begin{align*}
 \sum_{m|Q} \mu (m) \left\lfloor \frac{x}{m} \right\rfloor
 & =
 \sum_{m|Q} \mu (m) \left( \sum_{j=1}^{\infty} \chi_{jm} (x) \right)
 \\
 & =
 \sum_{m|Q} \sum_{j=1}^{\infty} \mu (m) \chi_{jm} (x)
 \\
 & =
 \sum_{u=1}^{\infty} \sum_{m|(u,Q)} \mu (m) \chi_{u} (x)
 \\
 & =
 \sum_{(u,Q)=1} \chi_{u} (x)
 ,
\end{align*}
we have the assertion of the lemma.
\end{proof}

\begin{cor}
\( d_i := (k+1)Q c_i \in \NN \).
\end{cor}
\begin{proof}
By {\bf{Lemma \ref{lemm:fdci}}}, the value of \( G_k (P;x) \) may
increase only when\\
\( \chi_{u} ((k+1)Qx) \) increases.
\end{proof}

Here we see some explicit examples.
We take \( P= \{ \; p \in \PP \; | \; p \leqq k \; \} \) in
all of the following examples.

\begin{itemize}
\item
{\bf{In the case of \( k=2 \):}}\\
\( F_2^* ( \{ 2 \} ;n) \) is certainly \( F_2^* (n) \) appeared before.
\item
{\bf{In the case of \( k=3 \):}}\\
As for
\[
 F_3^* ( \{ 2,3 \} ;n)= \frac{(24n)!(9n)!(4n)!(2n)!}{(18n)!(12n)!(8n)!n!}
 ,
\]
we have \( M=1, \; d_1 =13, \; c= \frac{24}{13} \) and
therefore the condition (\ref{eq:main}) holds,
because \( (\text{LHS})=16 \log 2 - 6 \log 3 >
(\text{RHS})= \frac{48}{13} \log 2 \).
\item
{\bf{In the case of \( k=4 \):}}\\
As for
\[
 F_4^* ( \{ 2,3 \} ;n) := \frac{
 (30n)!(12n)!(8n)!(5n)!(3n)!(2n)!
 }{
 (24n)!(15n)!(10n)!(6n)!(4n)!n!
 }
 ,
\]
we have \( M=1, \; d_1 = 13, \; c= \frac{30}{13} \) and
therefore the condition (\ref{eq:main}) holds,
because \( (\text{LHS})=10 \log 5 - 16 \log 2 >
(\text{RHS})= \frac{60}{13} \log 2 \).
\item
{\bf{In the case of \( k=5 \):}}\\
As for
\[
 F_5^* ( \{ 2,3,5 \} ;n) := \frac{
 (180n)!(75n)!(50n)!(30n)!(18n)!(12n)!n!
 }{
 (150n)!(90n)!(60n)!(36n)!(25n)!(3n)!(2n)!
 }
 ,
\]
we have \( M=2, \; d_1 =13, \; d_2 =49, \;
c= \frac{11160}{637} \).
Unfortunately, the condition (\ref{eq:main}) does not hold,
because \( (\text{LHS})=21.627 \cdots <
(\text{RHS})= \frac{11160}{637} \log 4 =24.287 \cdots \).
However, if we admit a bit sharper estimation of
the primorial \( \Theta (x) < 3.4^x \) instead
of {\bf{Lemma \ref{lemm:the}}},
this \( F_5^* \) proves \( P^* (5,30) \),
because \( \frac{11160}{637} \log 3.4 =21.440 \cdots \).
The same situation occurs in the case \( k=7,9,10 \).
\item
{\bf{In the case of \( k=6 \):}}\\
As for \( F_6^* ( \{ 2,3,5 \} ;n) \), we have \( M=2, \;
d_1 =13, \; d_2 =43 \) and
therefore the condition (\ref{eq:main}) holds,
because \( (\text{LHS})= 22.967 \cdots >
(\text{RHS})= 22.092 \cdots \).
\end{itemize}

The case \( k=5 \) imply that
we need another ingenuity for larger \( k \).

\subsection{Our ingenuity}

First we consider the case \( k=5 \).
In the last subsection,
we have \( d_1 =13 \) for \( F_5^* ( \{ 2,3,5 \} ;n) \).
But this \( d_1 \) is too small to satisfy the condition (\ref{eq:main}).
With reference to this fact,
we try the case \( F_5^* ( \{ 2,3,5,13 \} ;n) \) and then
we have the following explicit result.

\begin{itemize}
\item
{\bf{In the case of \( k=5 \):}}\\
As for \( F_5^* ( \{ 2,3,5,13 \} ;n) \),
we have
\( M=3, \; d_1 =19, \; d_2 =49, \;
d_3 =1309 \) and
therefore the condition (\ref{eq:main}) holds,
because \( (\text{LHS})= 259.523 \cdots >
(\text{RHS})= 239.414 \cdots \).
This method also work for \( k=7,9,10 \).
\end{itemize}

This method works well for some \( k \).

\begin{itemize}
\item
{\bf{In the case of \( k=8 \):}}\\
In this case we need to apply the above method twice.
Then, as for \( F_8^* ( \{ 2,3,5,7,19,31 \} ;n) \),
we have \( M=9, \; d_1 = 41, \; \dots \; , d_9 =785179 \) and
therefore the condition (\ref{eq:main}) holds,
because \( (\text{LHS})= 81375.551 \cdots >
(\text{RHS})= 73787.953 \cdots \).
\item
{\bf{In the case of \( k=11 \):}}\\
As for \( F_{11}^* ( \{ 2,3,5,7,11,31,43 \} ;n) \),
we have\\
\( M=11, \; d_1 =61, \; \dots \; , d_{11} =7544113 \) and
therefore the condition (\ref{eq:main}) holds,
because \( (\text{LHS})= 2081740.831 \cdots >
(\text{RHS})= 2067240.713 \cdots \).
\item
{\bf{In the case of \( k=12 \):}}\\
As for \( F_{12}^* ( \{ 2,3,5,7,11,31,43 \} ;n) \),
we have\\
\( M=13, \; d_1 =61 , \; \dots \; , d_{13} =3233107 \) and
therefore the condition (\ref{eq:main}) holds,
because \( (\text{LHS})= 2132199.327 \cdots >
(\text{RHS})= 1997810.591 \cdots \).
\end{itemize}

For larger \( k \), it's difficult
to have explicit values of \( d_i \)'s even with personal computers.
In the following example, we use the following lemma to
find an upper bound of (RHS).

\begin{lem}
\label{lemm:ubdg}
We have
\[
 |G_k \left( P;x \right) | \leqq 2^{\#P-1}
 .
\]
\end{lem}
\begin{proof}
Since \( G_k (P \cup \{ p \} ;x) = G_k (P;px) - G_k (P;x)
\) for \( p \not\in P \), we have the assertion of the lemma.
\end{proof}

\begin{itemize}
\item
{\bf{In the case of \( k=13 \):}}\\
As for \( F_{13}^* ( \{ 2,3,5,7,11,13,31,43 \} ;n) \),
we have\\
\( M \geqq 11, \; d_1 =61 , \; \dots \; , d_{11} =57859 \) and \(
d_{j} \geqq 1200000 \; (j>11) \) if \( M>11 \).
Therefore \( (\text{RHS}) \leqq Q \left( \frac{1}{d_1} + \dots
+ \frac{1}{d_{11}} + \frac{2^6 -11}{1200000} \right) \log 4
= 26139393.317 \cdots \).
Therefore the condition (\ref{eq:main}) holds,
because \( (\text{LHS})= 26145220.719 \cdots \).
\item
{\bf{In the case of \( k=14 \):}}\\
As for \( F_{14}^* ( \{ 2,3,5,7,11,13,31,43,61,71,83 \} ;n) \),
we have\\
\( M \geqq 28, \; d_1 =101 , \; \dots \; , d_{28} =611203 \) and \(
d_{j} \geqq 4000000 \; (j>28) \) if \( M>28 \).
Therefore \( (\text{RHS}) \leqq Q \left( \frac{1}{d_1} + \dots
+ \frac{1}{d_{28}} + \frac{2^{10} -28}{4000000} \right) \log 4
= 9173820030601.213 \cdots \).
Therefore the condition (\ref{eq:main}) holds,\\
because \( (\text{LHS})= 9183103103292.885 \cdots \).
\item
{\bf{In the case of \( k=15 \):}}\\
As for \( F_{15}^* ( \{ 2,3,5,7,11,13,43,61,71,83 \} ;n) \),
we have\\
\( M \geqq 24, \; d_1 =101 , \; \dots \; , d_{24} =1186213 \) and \(
d_{j} \geqq 20000000 \; (j>24) \) if \( M>24 \).
Therefore \( (\text{RHS}) \leqq Q \left( \frac{1}{d_1} + \dots
+ \frac{1}{d_{24}} + \frac{2^9 -24}{20000000} \right) \log 4
= 311652500411.160 \cdots \).
Therefore the condition (\ref{eq:main}) holds,\\
because \( (\text{LHS})= 311662041740.439 \cdots \).
\end{itemize}

Now we finish to give an entirely elementary proof
of the following theorem.
\begin{thm}
For any positive integer \( N \),
there exists a positive integer \( L_N \) such that
for any real number \( x \geqq L_N \),
there are at least \( N \) primes in the interval \( [15x, 16x] \).
\end{thm}

We remark that we have an entirely elementary proof
of \( P(k) \) for \( k \leqq 18 \),
by using a bit sharper estimation of
the primorial \( \Theta (x) < 3.4^x \) instead
of {\bf{Lemma \ref{lemm:the}}}.

\section*{Acknowledgments}
We thank Professor Pieter Moree for useful comment.
We also thank anonymous referee for a careful reading
of the manuscript and for many helpful comment.
This work is supported by
JSPS KAKENHI Grant Numbers 19K03429 and 23K03039.

\end{document}